\documentclass[a4paper]{article}
\usepackage{amsthm,amsmath,amssymb}
\usepackage{enumerate,enumitem}
\usepackage{graphicx}

\usepackage[all]{xy}
\usepackage{xfrac}
\usepackage{ulem}
\usepackage{tabularx,booktabs,multirow}
\usepackage{float}
\usepackage{hyperref}
\newcolumntype{C}{>{\centering\arraybackslash}X}
\newcolumntype{D}{>{\centering\arraybackslash}X}

\DeclareGraphicsExtensions{.pdf,.png,.eps}
\graphicspath{{figs/}}

\newtheorem{theorem}{Theorem}
\newtheorem{lemma}[theorem]{Lemma}
\newtheorem{observation}[theorem]{Observation}

\newtheorem{proposition}[theorem]{Proposition}
\newtheorem{corollary}[theorem]{Corollary}

\newtheorem*{claim*}{Claim}

\theoremstyle{remark}

\newcommand{\subsetneql}{\ensuremath{\subset}}

\newcommand{\cO}{\ensuremath{\mathcal{O}}}
\newcommand{\cT}{\ensuremath{\mathcal{T}}}
\newcommand{\N}{\ensuremath{\mathbb{N}}}
\newcommand{\cM}{\ensuremath{\mathcal{M}}}

\newcommand{\cX}{\ensuremath{\mathcal{X}}}
\newcommand{\cY}{\ensuremath{\mathcal{Y}}}
\newcommand{\cZ}{\ensuremath{\mathcal{Z}}}

\newcommand{\EEE}{\ensuremath{\mathbb{E}}}

\newcommand{\TT}{\ensuremath{\mathcal{T}}}
\newcommand{\PP}{\ensuremath{\mathcal{P}}}
\newcommand{\QQ}{\ensuremath{\mathcal{Q}}}
\newcommand{\PPP}{\ensuremath{\mathbb{P}}}

\usepackage[usenames,dvipsnames]{xcolor}

\begin{document}

\title{Poset Ramsey numbers: large Boolean lattice versus a fixed poset}

\author{Maria Axenovich\thanks{Karlsruhe Institute of Technology, Karlsruhe, Germany}\and Christian Winter\thanks{Karlsruhe Institute of Technology, E-mail: \textit{christian.winter@kit.edu}}}

\maketitle

\begin{abstract}

Given partially ordered sets (posets) $(P, \leq_P)$ and $(P', \leq_{P'})$, we say that $P'$ contains a copy of $P$ if for some injective function $f: P\rightarrow P'$  and for any $X, Y\in P$, $X\leq _P Y$ if and only of $f(X)\leq_{P'} f(Y)$. 
For any posets $P$ and $Q$, the poset Ramsey number $R(P,Q)$ is the least positive integer $N$ such that no matter how the elements of an $N$-dimensional Boolean lattice are colored in blue and red, there is either a copy of  $P$ with all blue elements or a copy of   $Q$ with all red elements. We focus on a poset Ramsey number $R(P, Q_n)$ for a fixed poset $P$ and an $n$-dimensional Boolean lattice $Q_n$, as $n$ grows large.
We show a sharp jump in behaviour of this number as a function of $n$ depending on whether or not $P$ contains a copy of either a poset $V$, 
i.e.\ a poset on elements $A, B, C$ such that $B>C$, $A>C$, and $A$ and $B$ incomparable, or a poset $\Lambda$, its symmetric counterpart. 
Specifically, we prove that if $P$ contains a copy of $V$ or $\Lambda$ then $R(P, Q_n) \geq n +\frac{1}{15} \frac{n}{\log n}$. Otherwise $R(P, Q_n) \leq n + c(P)$ for a constant $c(P)$. 
This gives the first non-marginal improvement of a lower bound on poset Ramsey numbers and as a consequence gives $R(Q_2, Q_n) = n + \Theta (\frac{n}{\log n})$.

\end{abstract}

\section{Introduction}

A partially ordered set, shortly a \textit{poset}, is a 
set $P$ equipped with a relation $\le_P$ that is transitive, reflexive, and antisymmetric. 
For any  non-empty set $\cX$, let $\QQ(\cX)$ be the \textit{Boolean lattice} of dimension $|\cX|$ on a \textit{ground set} $\cX$, i.e.\ 
the poset consisting of all subsets of $\cX$ equipped with the inclusion relation, $\subseteq$. 
We use $Q_n$ to denote a Boolean lattice with an arbitrary $n$-element ground set. 
We refer to a poset either as a pair $(P, \leq_P)$, or, when it is clear from context, simply as a set $P$. The elements of $P$ are often called \textit{vertices}.\\

For two posets $(P_1, \leq_{P_1})$ and $(P_2, \leq_{P_2})$,  an \textit{embedding} $\phi\colon P_1\to P_2$ of $P_1$ into $P_2$ is an injective function such that for every $X_1,X_2\in P_1$,
$$X_1 \leq_{P_1}  X_2\text{ if and only if }\phi(X_1)\leq_{P_2} \phi(X_2).$$
A poset $P_1$ is an  \textit{induced subposet} of $P_2$
if $P_1\subseteq P_2$ and for every $X_1,X_2\in P_1$, $X_1 \leq_{P_1}  X_2$ if and only if $X_1 \leq_{P_2} X_2.$
An  \textit{copy} of a poset $P_1$ in $P_2$ is an induced subposet $P'$ of $P_2$, isomorphic to $P_1$.\\

Extremal properties of posets and their induced subposets have been investigated in recent years and mirror similar concepts in graphs.
Carroll and Katona \cite{CK} initiated the consideration of so called Tur\'an-type problems for induced subposets. 
Most notable is a result by Methuku and P\'alvölgyi \cite{MP} which provides an asymptotically tight bound on the maximum size of a subposet of a Boolean lattice that does not have a copy of a fixed poset $P$, for general $P$. 
Their statement has been refined for several special cases, see e.g.\ Lu and Milans \cite{LM} and M\'eroueh \cite{M}.
Further Tur\'an-type results are, for example, given by Methuku and Tompkins \cite{MT} and Tomon \cite{T}. 
Note that Tur\'an-type properties are also investigated in depth for non-induced, so called \textit{weak} subposets, which are not considered here.
Besides that, saturation-type extremal problems are studied for induced and weak subposets, see a recent survey of Keszegh, et al.\ \cite{KLMPP}.
\\

In this paper we are dealing with Ramsey-type properties of induced subposets in Boolean lattices. 
Consider an assignment of two colors, blue and red, to the vertices of posets. Such a coloring $c: P \rightarrow \{blue, red\}$ is a \textit{blue/red coloring} of $P$.
A colored poset is \textit{monochromatic} if all of its vertices share the same color. A monochromatic poset whose vertices are blue is called a \textit {blue poset}. Similarly defined is a \textit{red poset}.
Extending the classical definition of graph Ramsey numbers, Axenovich and Walzer \cite{AW} introduced the \textit{poset Ramsey number} which is defined as follows. For posets $P$ and $Q$, let 

\begin{multline*}
R(P,Q)=\min\{N\in\N \colon \text{ every blue/red coloring of $Q_N$ contains either}\\ 
\text{a blue copy of $P$ or a red copy of $Q$}\}.
\end{multline*}

One of the central questions in this area is to determine $R(Q_n, Q_n)$. The best bounds currently known are 
$2n+1 \leq R(Q_n, Q_n) \leq n^2 -n+2$, see listed chronologically Walzer \cite{W},  Axenovich and Walzer \cite{AW},  Cox and Stolee \cite{CS},   Lu and Thompson \cite{LT},  Bohman and Peng \cite{BP}. 
For off-diagonal setting $R(Q_k, Q_n)$ with $k$ fixed and $n$ large, an exact result is only known if $k=1$. It is easy to see that $R(Q_1, Q_n)=n+1$.
For $k=2$, it was shown in \cite{AW} that $R(Q_2, Q_n) \leq 2n+2$. This was improved by Lu and Thompson to $R(Q_2, Q_n)\leq (5/3)n +2$. 
Finally, it was further improved by Gr\'osz, Methuku, and Tompkins \cite{GMT}:
\begin{theorem}[Gr\'osz et al. \cite{GMT}]\label{thm_grosz}
Let $\epsilon>0$ and let $n\in\N$ be sufficiently large. Then
$$n+3 \leq R(Q_2,Q_n) \le n + \frac{(2+\epsilon)n}{\log n}.$$
\end{theorem}

Further known bounds on poset Ramsey numbers include results of Chen et al. \cite{CCCLL}, \cite{CCLL} as well as Chang et al. \cite{CGLMNPV}.
\\

In this paper,  we start a more systematic investigation of $R(P, Q_n)$ for a fixed poset $P$ and large $n$. This Ramsey number gives an analogue of the graph Ramsey number $R(H,K_n)$ that claims that every edge-coloring of a complete graph in red and blue with no given induced blue subgraph $H$, contains a large red clique of size $n$. 
Here, the goal is to provide a quantitative version of a  statement that every blue/red coloring of a Boolean lattice with no blue (induced) subposet $P$ contains a large red Boolean sublattice.  
One of the key roles here plays a small, three-vertex poset $\Lambda = (\Lambda, <)$,  with vertices $Z_1, Z_2$ and $Z_3$, such that $Z_1<Z_3$, $Z_2<Z_3$, and $Z_1$ and $Z_2$ incomparable.
A poset $V$ is the symmetric counterpart of $\Lambda$, having vertices $Z_1, Z_2 $ and $Z_3$, such that $Z_1>Z_3$, $Z_2>Z_3$, and $Z_1$ and $Z_2$  not comparable.

Our main result shows a sharp jump in the behaviour of $R(P, Q_n)$ as a function of $n$ depending whether or not $P$ contains a copy $\Lambda$ or $V$.

\begin{theorem}\label{thm-MAIN}
For every poset $P$ there is an integer $n_0$ such that for any $n>n_0$ the following holds.
If $P$ contains a copy of $\Lambda$ or $V$, then $R(P, Q_n) \geq n + \frac{1}{15} \frac{n}{\log n}$.
If $P$ contains neither a copy of $\Lambda$ nor a copy of $V$, then $R(P, Q_n) \leq n + f(P)$, for some function $f$.
\end{theorem}

In order to show the lower bound, we prove a structural duality statement that together with a probabilistic construction allows to find a desired coloring. 
This is the first of a kind non-marginal improvement of a trivial lower bound for poset Ramsey numbers. 
Most other known lower bounds corresponded to so-called layered colorings of Boolean lattices, where any two vertices of the same size have the same color.
The only two constructions different from this and known so far are the aforementioned lower bound of Gr\'osz et al. \cite{GMT} 
as well as a construction of Bohman and Peng \cite{BP} improving the trivial lower bound for the diagonal case $R(Q_n,Q_n)\ge 2n$ to $2n+1$.\\

We show  the following bounds on the poset Ramsey number of $\Lambda$ versus $Q_n$.

\begin{theorem}\label{thm_main}Let $\epsilon>0$. There exists an $n_0\in\N$ such that for all $n\ge n_0$,
$$n+ \frac{1}{15}\cdot\frac{n}{\log n}\leq R(\Lambda,Q_n)\le n+ \big(1+\epsilon)\cdot\frac{n}{\log n}.$$
\end{theorem}

More precisely, it can be seen that the lower bound holds for $\log n_0\ge535$, while the upper bound requires $\log n_0\ge \frac{36}{\epsilon^2}$.
Note that $R(\Lambda,Q_n)\le R(Q_2,Q_n)$, so Theorem \ref{thm_grosz} already implies a bound for $R(\Lambda,Q_n)$ 
which is weaker but asymptotically equal to the upper bound of Theorem \ref{thm_main}.
In this paper we provide 
a duality statement that allows us to prove both the lower and the improved upper bound on $R(\Lambda, Q_n)$. 
Theorem \ref{thm_grosz} and Theorem \ref{thm-MAIN} also give a lower bound for $R(Q_2, Q_n)$ which is asymptotically tight not only in the first but also in the second summand.

\begin{corollary}\label{cor_Q2}
$$R(Q_2,Q_n) = n + \Theta\left(\frac{n}{\log (n)}\right).$$
\end{corollary}

The structure of the paper is as follows.
In Section \ref{sec_prelim} we introduce some notations and a new type of poset and show some useful propositions.
In Section \ref{sec_grosz} we provide an alternative proof of the upper bound in Theorem \ref{thm_grosz}. 
This makes our paper self-contained since we need this result for Corollary \ref{cor_Q2}.
In Section \ref{sec_duality} we provide a structural duality statement, Theorem \ref{thm_duality}, which is the key tool for the main proofs.
In Section \ref{sec_random} we use a probabilistic construction to find a coloring with ``good'' properties.
Lastly, in Section \ref{sec_proof} we complete the proofs of Theorem \ref{thm_main} and Theorem \ref{thm-MAIN}.
\\


\section{Preliminaries}\label{sec_prelim}

\subsection{Basic Notation}
Let $\cX$ and $\cY$ be disjoint sets. 
Then the vertices of the Boolean lattice $\QQ(\cX\cup\cY)$, i.e.\ the unordered subsets of $\cX\cup\cY$, can be partitioned with respect to $\cX$ and $\cY$ in the following manner.
Every $Z\subseteq\cX\cup\cY$ has an $\cX$-part $X_Z=Z\cap \cX$ and a $\cY$-part $Y_Z=Z\cap\cY$.
In this setting, we refer to $Z$ alternatively as the pair $(X_Z,Y_Z)$. 
Conversely, for all $X\subseteq\cX$, $Y\subseteq\cY$, the pair $(X,Y)$ has a $1$-to-$1$ correspondence to the vertex $X\cup Y\in \QQ(\cX\cup\cY)$.
One can think of such pairs as elements of the Cartesian product $2^\cX \times 2^\cY$ which has a canonic bijection to $2^{\cX\cup\cY}=\QQ(\cX\cup\cY)$.
Observe that for $X_i\subseteq\cX, Y_i\subseteq\cY$, $i\in[2]$, we have $(X_1,Y_1)\subseteq (X_2,Y_2)$ if and only if $X_1\subseteq X_2$ and $Y_1\subseteq Y_2$. We omit floors and ceilings where appropriate.
\\

For any poset, we refer to vertices $Z_1,Z_2$ which are incomparable as $Z_1\nsim Z_2$.
For a positive integer $n\in\N$, we use $[n]$ to denote the set $\{1,\dots,n\}$. Given an integer $n\in\N$ and a set $\cX$, let $\binom{\cX}{n}$ be the set of all $n$-element subsets of $\cX$.
Throughout the paper, `$\log$' always refers to the logarithm with base $2$, while `$\ln$' refers to the natural logarithm.
\\ \bigskip

\subsection{Structure of posets with forbidden $\Lambda$ or $V$}
A poset $\TT$ is an \textit{up-tree} if there is a unique minimal vertex in $\TT$ and for every vertex $X\in\TT$, the set $\{Y\in\TT\colon Y\le X\}$ is a chain, i.e.,\ its vertices are pairwise comparable.
We say that two posets are \textit{independent} if they are vertexwise incomparable. Furthermore, a collection of posets is \textit{independent} if they are pairwise independent.

We use this notation to describe posets which don't contain a copy of $\Lambda$ (or $V$). 

\begin{lemma}\label{strucL}
Let $P$ be a poset. There is no copy of $\Lambda$ in $P$ if and only if $P$ is an independent collection of up-trees.
\end{lemma}

\begin{proof} Observe that a poset $P$ is an independent collection of up-trees if and only if for every vertex $X\in P$, $\{Y\in P\colon Y\le X\}$ forms a chain.

Suppose that there is a copy of $\Lambda$ in $P$ on vertices $Z_i$, $i\in[3]$ with $Z_1<Z_3$, $Z_2<Z_3$ and $Z_1\nsim Z_2$.
Then $Z_1,Z_2$ witness that $\{Y\in P\colon Y<Z_3\}$ is not a chain, so $P$ is not an independent collection of up-trees.

Now assume that $P$ is not an independent collection of up-trees. 
Then there exist some $X\in P$ and $Z_1,Z_2 \in \{Y\in P\colon Y\le X\}$ such that $Z_1\nsim Z_2$. 
Since $X$ is comparable to all vertices in $\{Y\in P\colon Y\le X\}$, $X> Z_1, X> Z_2$. Now $X,Z_1,Z_2$ form a copy of $\Lambda$.
\end{proof}

By symmetry an analogous statement holds for posets with forbidden induced copy of $V$.
If we forbid both $V$ and $\Lambda$ simultaneously we obtain the following structure.
\begin{corollary}\label{strucI}
Let $P$ be a poset such that there is neither a copy of $V$ nor of $\Lambda$.
Then $P$ is an independent collection of chains.
\end{corollary}
\bigskip

\subsection{Embeddings of  $Q_n$}
When considering an embedding $\phi$ of a Boolean lattice $Q_n$ into a larger Boolean lattice $\QQ(\cZ)$, we can partition $\cZ$ such that it has the following nice property.

\begin{lemma}\label{embed_lem}
Let $n\in\N$. Let $\cZ$ be a set with $|\cZ|>n$ and let $Q=\QQ(\cZ)$. If there is an embedding $\phi\colon Q_n\to Q$,
then there exist a subset $\cX\subsetneql\cZ$ with $|\cX|=n$, and an embedding $\phi'\colon \QQ(\cX)\to Q$ such that
 $\phi'(X)\cap \cX=X$ for all $X\subseteq \cX$.
\end{lemma}
\begin{proof}Let the ground set of $Q_n$ be $\cX'$.
We consider the embedding of singletons of $Q_n$, i.e.\ $\phi(\{a\})$, $a\in\cX'$.
If $\phi(\{a\})\subseteq \bigcup_{X'\subseteq \cX'\backslash\{a\}} \phi(X')$, 
then $\phi(\{a\})\subseteq \bigcup_{X'\subseteq \cX'\backslash\{a\}} \phi(X') \subseteq \phi(\cX'\backslash\{a\})$. 
But $\{a\}\nsubseteq \cX'\backslash\{a\}$ and $\phi$ is an embedding, a contradiction. Thus $\phi(\{a\})\not\subseteq \bigcup_{X'\subseteq \cX'\backslash\{a\}} \phi(X')$.
For every $a\in\cX'$, pick an arbitrary $$b(a)\in \phi(\{a\})\backslash\bigcup_{X'\subseteq \cX'\backslash\{a\}} \phi(X').$$
Note that $b(a_1)\notin \phi(\{a_2\})$ for any $a_1,a_2\in\cX'$, $a_1\neq a_2$, so all representatives are distinct.
Let $\cX=\{b(a) \colon a\in\cX'\}$. We see that  the map $b: \cX' \to \cX$ is a bijection.
For every   $B\subseteq \cX$, let $A_B\subseteq \cX'$ be such that $B=\{b(a): a \in A_B\}$. 
We define $\phi'\colon \QQ(\cX)\to Q$ as follows:  $\phi'(B) = \phi(A_B)$, $B\in \cX$.
Then $\phi'$ is an embedding.
Observe that for $X\subseteq \cX$ and $b\in\cX$, \ $b\in \phi'(X)$ if and only if $b\in X$. Thus $\phi'(X)\cap \cX=X$ for all $X\subseteq \cX$.
This concludes the proof.
\end{proof}

We call an embedding $\phi$  of $\QQ(\cX)$ into $\QQ(\cX\cup \cY)$ for disjoint $\cX$ and $\cY$, \textit{$\cX$-good} if $\phi(X)\cap \cX = X$  for all $X\in \cX$. 
We also call a copy $Q$ of $\QQ(\cX)$ in $\QQ(\cX\cup \cY)$ $\cX$-\textit{good} if there is an $\cX$-good embedding of $\QQ(\cX)$ into $\QQ(\cX\cup \cY)$.

Lemma \ref{embed_lem} claims in particular that for any copy of $Q_n$ in a larger Boolean lattice $Q$, 
there is a subset $\cX$ of the ground set of $Q$ with $|\cX|=n$ such that there is an $\cX$-good copy of $\QQ(\cX)$ in $Q$. 
\\
\bigskip

\subsection{Red copy of $Q_n$ vs.\ blue chain}\label{sec_grosz}

The main goal of this subsection is to present an alternative proof for the upper bound of Theorem \ref{thm_grosz}.
Gr\'osz, Methuku, and Tompkins \cite{GMT} stated the following lemma using a different formulation.
While they used algorithmic tools in their proof, we prove the statement recursively. 
Recall that for a given partition $\cX\cup \cY$ of the ground set of a Boolean lattice we denote a vertex $X\cup Y\in\QQ(\cX\cup\cY)$, where $X\subseteq\cX, Y\subseteq\cY$, as $(X,Y)$.

\begin{lemma}
\label{shift_lem}
Let $\cX$, $\cY$ be disjoint sets with $|\cX|=n$ and $|\cY|=k$, for some $n,k\in\N$. Let $Q=\QQ(\cX\cup\cY)$ be a blue/red colored Boolean lattice.
Fix some linear ordering  $\pi=(y_1,\dots,y_k)$  of $\cY$ and define $Y(0), \ldots, Y(k)$ by $Y(0)=\varnothing$ and  $Y(i)=\{y_1,\dots,y_i\}$ for $i\in[k]$.
Then there exists at least one of the following in $Q$:
\renewcommand{\labelenumi}{(\alph{enumi})}
\begin{enumerate}
\item a red $\cX$-good copy of $\QQ(\cX)$, or
\item a blue chain of length $k+1$ of the form $(X_0,Y(0)),\dots,(X_{k},Y(k))$ where $X_0 \subseteq X_1 \subseteq\dots \subseteq X_k\subseteq \cX$.
\end{enumerate}
\end{lemma}

\begin{proof}
Suppose that there is no blue chain as described in (b).
For every $X\subseteq\cX$, we recursively define a label $\ell_X\in\{0,\dots,k\}$ such that $\phi\colon \QQ(\cX)\to Q$, $\phi(X)=(X,Y(\ell_X))$, is an embedding with monochromatic red image. 
We require $\ell_X$ to fulfill three properties: 
\begin{enumerate}[label=(\arabic*)]
\item For any $X'\subseteq X$, \ $\ell_{X'}\le \ell_{X}$.
\item There is a blue chain of length $\ell_X$ contained 
in the 
Boolean lattice
with ground set $X\cup Y(\ell_X)$, which we denote by $Q^{X}$.
\item $(X,Y(\ell_X))$ is red.
\end{enumerate}

First, consider the vertex $\varnothing$. Let $\ell_\varnothing$ be the minimum $\ell$, $0\le \ell\le k$, such that $(\varnothing,Y(\ell))$ is red.
If such an $\ell$ does not exist, then $(\varnothing,Y(0)), \dots, (\varnothing,Y(k))$ form a blue chain, a contradiction. 
It is clear to see that Properties (1) and (3) hold.
If $\ell_\varnothing=0$, (2) is trivially true. 
If $\ell_\varnothing\ge 1$, $(\varnothing,Y(0)), \dots, (\varnothing,Y(\ell_{\varnothing}-1))$ form a blue chain of length $\ell_\varnothing$ and $(2)$ holds as well.
\\

Consider an arbitrary $X\subseteq\cX$ and suppose that for all $X'\subsetneql X$ we already defined $\ell_{X'}$ with Properties (1)-(3). Let $\ell'_X=\max_{\{U\subsetneql X\}} \ell_U$.
Then let $\ell_X$ be the minimum~$\ell$, $\ell'_X\le \ell \le k$ such that $(X,Y(\ell))$ is colored in red.
If there is no such $\ell$, then $(X,Y(\ell'_X)),\dots,(X,Y(k))$ is a blue chain of length $k-\ell'_X+1$. By definition of $\ell'_X$ there is some $U\subsetneql X$ with $\ell_U=\ell'_X$.
In particular, $(2)$ holds for $U$, so there is a blue chain of length $\ell'_X$ in $Q^{U}$. Note that $(U,Y(\ell_U))\subsetneql(X,Y(\ell'_X))$, 
so we obtain a blue chain of length $k+1$. This is a contradiction, thus $\ell_X$ is well-defined and fulfills Property (3). 

If $\ell_X=\ell'_X$, consider the aforementioned blue chain of length $\ell'_X$ in $Q^{U}$, and otherwise consider this chain together with $(X,Y(\ell'_X)), \dots, (X,Y(\ell_X-1))$.
In both cases, we obtain a blue chain of length $\ell_X$, which proves $(2)$.
For $X'\subsetneql X\subseteq \cX$, \ $\ell_{X'}\le \ell'_X\le \ell_X$, thus (1) holds.
\\

%
%
%

We define $\phi\colon \QQ(\cX)\to Q$, $\phi(X)=(X,Y(\ell_X))$. Note that $\phi(X)\cap\cX=X$ for every $X\subseteq\cX$ and Property (3) implies that $\phi(X)$ is red.
Let $X_1,X_2\subseteq\cX$. If $\phi(X_1)\subseteq\phi(X_2)$, it is immediate that $X_1\subseteq X_2$.
Conversely, if $X_1\subseteq X_2$, then by Property (1) we have $\ell_{X_1}\le \ell_{X_2}$. Thus $(X_1,Y(\ell_{X_1}))\subseteq(X_2,Y(\ell_{X_2}))$.
As a consequence, $\phi$ is an $\cX$-good embedding of $\QQ(\cX)$.
\end{proof}

This Lemma implies the following corollary which is already given in an alternative form by Axenovich and Walzer, see Lemma 4 of \cite{AW}.

\begin{corollary}\label{shift_cor}
Let $\cX,\cY$ be disjoint sets with $|\cX|=n$ and $|\cY|=k$. 
Let $\PP$ be a subposet of a Boolean lattice $Q=\QQ(\cX\cup\cY)$ such that there is no chain of length $k+1$ in $\PP$. 
Then there exists a copy of $Q_n$ in $Q$ which contains no vertex of $\PP$.
\end{corollary}

\begin{proof} Fix an arbitrary linear ordering of $\cY$.
Furthermore, let $c: Q \to \{blue,red\}$ be the coloring such that
\begin{equation} \nonumber
c(X) = 
 \begin{cases}
\text{blue},  		\hspace*{0.7cm}\mbox{ if } ~~  X\in \PP,\\
\text{red}, 		\hspace*{0.8cm}\mbox{ otherwise.}
\end{cases}
\end{equation}
There is no blue chain of length $k+1$ in $c$, so by Lemma \ref{shift_lem} there is a monochromatic red copy of $Q_n$ in $Q$.
This copy does not contain any vertex of $\PP$.
\end{proof}

With the help of Lemma \ref{shift_lem}, we can now prove an upper bound for $R(Q_2,Q_n)$. The concluding arguments are due to Gr\'osz, Methuku, and Tompkins \cite{GMT}.

\begin{proof}[Proof of Theorem \ref{thm_grosz}]
For the lower bound, the reader is referred to \cite{GMT}. 
For the upper bound, let $k\in\N$ with $k=\tfrac{(2+\epsilon)n}{\log (n)}$. 
Let $\cX$ and $\cY$ be disjoint sets with $|\cX|=n$ and $|\cY|=k$. Consider a blue/red coloring of $Q=\QQ(\cX\cup\cY)$ with no monochromatic red copy of $Q_n$.
Let $\pi=(y_1^\pi,\dots,y_k^\pi)$ be a linear ordering of $\cY$.
By Lemma \ref{shift_lem}, there exists a blue chain of length $k+1$ 
of the form $(X^\pi_0,\varnothing), (X^\pi_1,\{y^\pi_1\}),(X^\pi_2,\{y^\pi_1,y^\pi_2\}),\dots, (X^\pi_k,\cY)$ where $X^\pi_i\subseteq\cX$.
\\

Note that there are $k!$ distinct orderings of $\cY$. 
For each linear ordering $\pi$ of $\cY$ we consider $X^\pi_0$ and $X^\pi_k$, i.e.\ the minimal and maximal vertex of the aforementioned chain restricted to $\cX$.
By the choice of $k$, we obtain $k!> 2^{2n}$. In particular by pigeonhole principle, there are distinct $\pi_1,\pi_2$ 
with $X^{\pi_1}_0=X^{\pi_2}_0$ and $X^{\pi_1}_k=X^{\pi_2}_k$.
Since $\pi_1,\pi_2$ are distinct, there exists an index $1\le i\le k-1$ with $\{y^{\pi_1}_1,\dots,y^{\pi_1}_i\}\neq \{y^{\pi_2}_1,\dots,y^{\pi_2}_i\}$.
Then the four vertices 
$$(X^{\pi_1}_0,\varnothing), (X^{\pi_1}_i,\{y^{\pi_1}_1,\dots,y^{\pi_1}_i\}), (X^{\pi_2}_i,\{y^{\pi_2}_1,\dots,y^{\pi_2}_i\}), (X^{\pi_1}_k,\cY)$$ form a blue copy of $Q_2$.
\end{proof}
\bigskip

\subsection{Factorial trees and shrubs}
Besides the Boolean lattice, there is another poset which plays a major role in this paper, which we call the \textit{factorial tree}.\\

Consider the set of ordered subsets of a fixed non-empty set $\cY$, that also could be thought of as a set of strings with non-repeated letters over the alphabet $\cY$. 
Note that we also allow the empty set as such an ordered subset. 
Occasionally, if it is clear from the context, we refer to the empty ordered set $(\varnothing,\le)$ simply as $\varnothing$.
For an ordered subset $S$ of $\cY$, we refer to its underlying unordered set as $\underline{S}$. Let $|S|=|\underline{S}|$ be the \textit{size} of $S$.
We also say that $S$ is an \textit{ordering} of $\underline{S}$. 
\\

Let $S$ be an ordered subset of $\cY$. A \textit{prefix} of $S$ is an ordered subset $T$ of $\cY$ consisting of the first $|T|$ elements of $S$ in the ordering induced by $S$.
If $T$ is a prefix of $S$, we write $T\le_\cO S$. Note that the empty ordered set is a prefix of every ordered set. 
If $T\neq S$, we say that a prefix $T$ of $S$ is \textit{strict}, denoted by $T<_\cO S$. 
Observe that the prefix relation $\le_\cO$ is transitive, reflexive and antisymmetric. Let $\cO(\cY)$ be the poset of all ordered subsets of $\cY$ equipped with $\le_\cO$. 
We say that this poset is the \textit{factorial tree} on ground set $\cY$.  
\\

In a factorial tree $\cO(\cY)$ for every vertex $S\in\cO(\cY)$, the set of prefixes $\{T\in\cO(\cY)\colon T\le_\cO S\}$ induces a chain.  Furthermore, the vertex $\varnothing$ is the unique minimal vertex of $\cY$, thus $\cO(\cY)$ is an up-tree. 
\\

Let $\cX$ and $\cY$ be disjoint sets. Let $Q=\QQ(\cX\cup\cY)$ and $\cO(\cY)$ be the factorial tree with ground set $\cY$. 
An embedding $\tau$ of $\cO(\cY)$ into $Q$ is \textit{$\cY$-good} if for every $S\in \cO(\cY)$, $\tau(S)\cap \cY=\underline{S}$.
We say that a subposet $\PP$ of $Q$ is a \textit{$\cY$-good copy} of $\cO(\cY)$ if
there exists a $\cY$-good embedding $\tau: \cO(\cY) \to Q$ with image $\PP$. We refer to such a copy also as a \textit{$\cY$-shrub}.\\

Besides that, we also consider a related subposet with slightly weaker conditions.
A \textit{weak $\cY$-shrub} is a subposet $\PP$ of $Q$ such that there is a function $\tau: \cO(\cY) \to Q$ with image $\PP$ 
such that for every $S\in\cO(\cY)$, \  $\tau(S)\cap \cY=\underline{S}$ and for every $S,T\in \cO(\cY)$ with $S<_\cO T$, \ $\tau(S)\subsetneql \tau(T)$. 
In particular, a weak $\cY$-shrub might not correspond to an injective embedding of  $\cO(\cY) $.
\\

Clearly a $\cY$-shrub is also a weak $\cY$-shrub. Surprisingly, the converse statement is also true in Boolean lattices not containing a copy of $\Lambda$.

\begin{proposition}\label{obs_shrub}
Let $\cX$ and $\cY$ be disjoint sets, let $Q=\QQ(\cX\cup\cY)$.
Let $\PP$  be a weak $\cY$-shrub in $Q$ such that $\PP$ contains no copy of $\Lambda$. 
Then $\PP$ is a $\cY$-shrub.
\end{proposition}

\begin{proof}
Let $\tau: \cO(\cY) \to Q$ be a map such that for every $S,T\in \cO(\cY)$ with $S<_\cO T$, we have  $\tau(S)\subsetneql \tau(T)$ and $\tau(S)\cap \cY=\underline{S}$, and let $\PP$ be its image. 
For all $S\in\cO(\cY)$, let $X_S=\tau(S)\cap \cX$, i.e.\ $\tau(S)=(X_S,\underline{S})$.
We shall show that $\tau$ is an embedding, thus proving that $\PP$ is a $\cY$-shrub. For that we need to prove that the condition $\tau(S)\subseteq\tau(T)$ implies that $S\le_\cO T$
for any ordered subsets $S$ and $T$ of $\cY$.
\\

Let $\tau(S)\subseteq\tau(T)$, i.e.\ $(X_S,\underline{S})\subseteq(X_{T},\underline{T})$.
In particular, $\underline{S}\subseteq \underline{T}$ and so $|S|\le |T|$.
Let $R$ be the largest common prefix of $S$ and $T$. Such exists since $\varnothing$ is a prefix of every ordered set.
If $|R|=|S|$, then 
$S=R\le_\cO T$ and we are done. So we can assume that $|S|\ge |R|+1$.\\

If $|T|\le |R|+1$, then $|R|+1\le |S|\le |T|\le |R|+1$. This implies $|S|=|T|$ and  since $\underline{S}\subseteq \underline{T}$, we have $\underline{S}=\underline{T}$.
Let $\{y\}= \underline{S}\backslash \underline{R}=\underline{T}\backslash \underline{R}$.
Then both $S,T$ have $R$ as prefix of size $|S|-1=|T|-1$ and $y$ as final vertex. Thus $S=T$ and we are done as well.\\

From now on, we assume that $|S|\ge |R|+1$ and $|T|> |R|+1$.
Consider prefixes $S'\le_\cO S$ and $T'\le_\cO T$ of size $|R|+1$. 
Then $R$ is a prefix of both $S'$ and $T'$. 
Let $y_S$ such that $\underline{S'}\backslash \underline{R}=\{y_S\}$ and let $y_T$ with $\underline{T'}\backslash \underline{R}=\{y_T\}$.

If $y_S=y_T$, we obtain $S'=T'$, which implies that $R$ is not the largest common prefix of $S$ and $T$, a contradiction.

If $y_S\neq y_T$, the unordered sets $\underline{S'}$ and $\underline{T'}$ are not comparable. In particular, $(X_{T'},\underline{T'})$ and $(X_{S'},\underline{S'})$ are incomparable.
Because $S'\le_\cO S$, $T'<_\cO T$ and by our initial assumption, we know that $(X_{S'},\underline{S'})\subseteq (X_S,\underline{S})\subseteq(X_{T},\underline{T})$ and $(X_{T'},\underline{T'})\subseteq (X_{T},\underline{T})$. 
Since $|S'|=|T'|=|R|+1<|T|$, we obtain that both $(X_{S'},\underline{S'})$ and $(X_{T'},\underline{T'})$ are proper subsets of $(X_{T},\underline{T})$.
Then the three vertices $(X_{T},\underline{T}), (X_{T'},\underline{T'})$ and $(X_{S'},\underline{S'})$ form a copy of $\Lambda$ in $Q$, so we reach a contradiction.
\end{proof}

\subsection{Construction of an almost optimal shrub}

Let $\cY$ be a $k$-element set. Note that a $\cY$-shrub has $k!$ maximal vertices corresponding to all permutations of $\cY$. 
These maximal vertices form an antichain, i.e.\ are pairwise incomparable. 
Sperner's theorem implies that a ground set of any $\cY$-shrub must have size at least $q$, where $\binom{q}{\lfloor q/2 \rfloor} \geq k!$, so $q\ge k(\log k +\log e) + o(k)$.
Next, we shall construct a $\cY$-shrub which is almost optimal in the sense that $\cY$ has ground set of size almost matching the lower bound above.

\begin{proposition} \label{canonical-shrub}
Let $\cY$ be a $k$-element set.
Let $A$ be a set disjoint from $\cY$ such that $|A| \geq k \cdot \min\{\log k + \log\log k,11\}$.
Then there is a $\cY$-shrub in $\QQ(A\cup\cY)$.
\end{proposition}

\begin{proof}
Let $\cY=\{y_0, \ldots, y_{k-1}\}$ and let $Q=\QQ(A\cup\cY)$.
We use addition of indices modulo $k$.
Let $A_0, \ldots, A_{k-1}$ be pairwise disjoint subsets of $A$  such that  $|A_i|= \ell$ for the smallest integer $\ell$ satisfying
$\binom{\ell}{\lfloor \ell/2 \rfloor} \geq  k$. 
Since $\ell\le \log k +\log\log k$ for $k\ge 256$ and $\ell\le 11$ for $k\le256$, such $A_i$'s can be chosen.
Let  $\big\{A_i^j: j\in \{0, \ldots, k-1\}\big\}$ be an antichain in $\QQ(A_i)$, $i\in\{0,\ldots, k-1\}$. 
Such an antichain exists by Sperner's theorem.\\

Consider the factorial tree $\cO(\cY)$. We shall construct an embedding $\tau$ of  $\cO(\cY)$ into $Q$ as follows. 
Let $\tau(\varnothing)= \varnothing$. 
Consider any non-empty ordered subset of $\cY$, say $(y_{i_1}, y_{i_2}, \ldots, y_{i_j})$, $1\leq j\leq k$. 
If $j=1$, let $\tau((y_{i_1}))=A_{i_1}\cup \{y_{i_1}\}$. If $j>1$, let 
 $$\tau( (y_{i_1}, \ldots, y_{i_j})) = A_{i_1}\cup A_{i_1+1}^{i_2} \cdots \cup A_{i_1+j-1}^{i_j} \cup \{y_{i_1}, \ldots, y_{i_j}\}.$$
For example for $k=4$, $\tau((y_0, y_1, y_2)) = A_0\cup A_1^1\cup A_2^2 \cup \{y_1, y_2, y_3\}$,  
$\tau((y_2, y_3, y_1)) = A_2 \cup A_3^3 \cup A_0^1 \cup \{y_1, y_2, y_3\}$, and $\tau((y_3,  y_1)) = A_3 \cup A_0^1  \cup \{y_1, y_3\}$.\\

Note that the image of $\cO(\cY)$ under $\tau$ is an up-tree, $\cT$, whose minimum vertex is $\varnothing$, see Figures \ref{fig:shrub}  and \ref{fig:shrub-1}. 
We see that each maximum vertex of $\cT$ is joined to $\varnothing$ by a unique chain, a maximal chain. 
Furthermore, non-zero vertices that belong to distinct maximal chains are incomparable.
Observe that $\tau$ is a $\cY$-good embedding of $\cO(\cY)$ into $Q$. 
Indeed, for any ordered sequence  of distinct vertices $(y_{i_1}, \ldots, y_{i_j})$, we have 
$\tau( (y_{i_1}, \ldots, y_{i_j}))\cap \cY = \{y_{i_1}, \ldots, y_{i_j}\}$.  In addition  $(y_{i_1}, \ldots, y_{i_q})<_{\cO} (y_{i_1}, \ldots, y_{i_p})$ if and only if 
 $\tau((y_{i_1}, \ldots, y_{i_q}))$ and  $\tau ((y_{i_1}, \ldots, y_{i_p}))$  are in the same maximal chain of $T$ in the corresponding order. 
\end{proof}


\begin{figure}
\includegraphics[width=\textwidth]{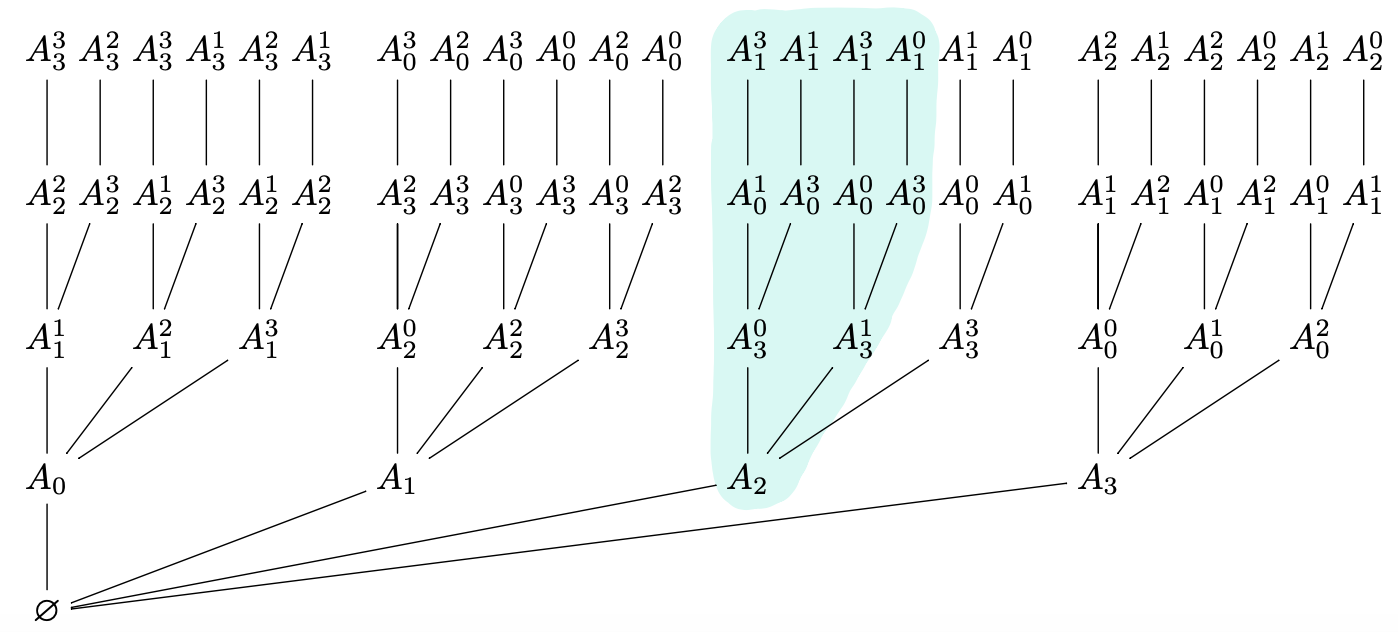}
\caption{A diagram illustrating how the $A_i^j$'s  are being assigned to the elements of the $\{y_0, y_1, y_2, y_3\}$-shrub constructed in Proposition \ref{canonical-shrub}.}
\label{fig:shrub}
\end{figure}

\begin{figure}
\includegraphics[width=\textwidth]{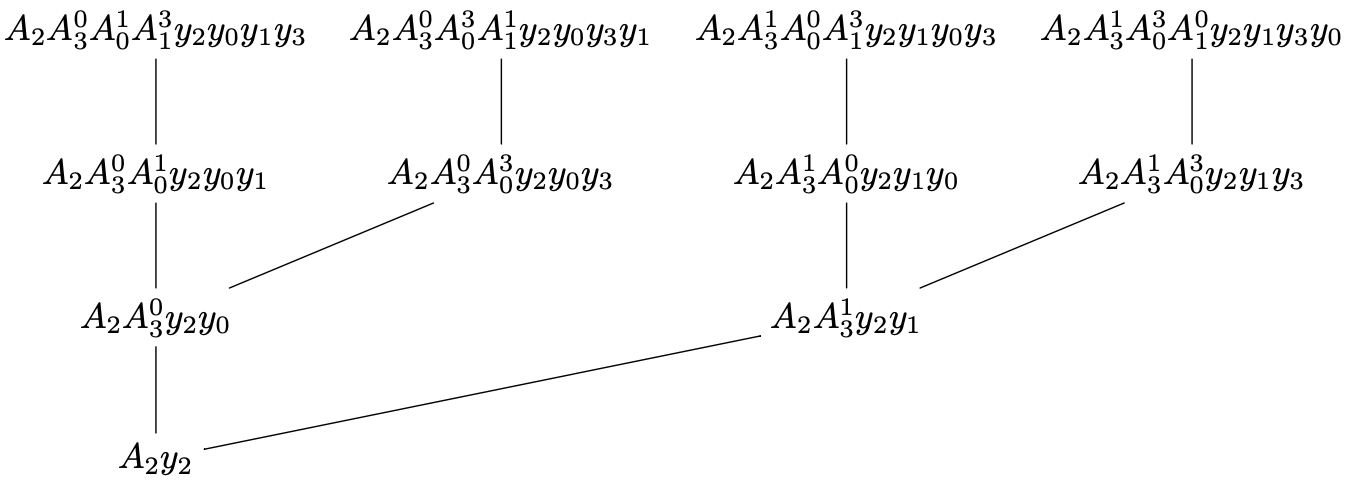}
\caption{Segment of a shrub highlighted in Figure \ref{fig:shrub}.  Here the union signs are omitted because of the spacing, for example $A_2  A_3^1  A_0^0 y_2 y_1 y_0$ corresponds to 
the shrub vertex $A_2 \cup  A_3^1 \cup A_0^0\cup \{ y_2, y_1, y_0\}$.}
\label{fig:shrub-1}
\end{figure}

\section{Duality theorem}\label{sec_duality}

In this section, we show a duality statement which is the key argument for the proof of Theorem \ref{thm_main}. Recall the following definitions. 
We call an embedding $\phi$ of $\QQ(\cX)$ into $\QQ(\cX\cup \cY)$ for disjoint $\cX$ and $\cY$, \textit{ $\cX$-good} if $\phi(X)\cap \cX = X$  for all $X\subseteq \cX$. 
We also call a copy of $\QQ(\cX)$ in $\QQ(\cX\cup \cY)$ $\cX$-\textit{good} if there is an $\cX$-good embedding of $\QQ(\cX)$ into $\QQ(\cX\cup \cY)$.
An embedding $\tau$ of $\cO(\cY)$ into $\QQ(\cX\cup \cY)$ is \textit{ $\cY$-good} if $\tau(S)\cap\cY=\underline{S}$ for all $S\in\cO(\cY)$.
We say that a copy of $\cO(\cY)$ is a \textit{$\cY$-shrub} if there exists a $\cY$-good embedding of $\cO(\cY)$ into $\QQ(\cX\cup \cY)$.

\begin{theorem}[Duality Theorem]\label{thm_duality}
For two disjoint sets $\cX$ and $\cY$, let $Q=\QQ(\cX\cup \cY)$ be a blue/red colored Boolean lattice which contains no blue copy of $\Lambda$.
Then there is exactly one of the following in $Q$: 
\begin{itemize}
\item a red $\cX$-good copy of $\QQ(\cX)$, or
\item a blue $\cY$-good copy of $\cO(\cY)$, i.e.\ a blue $\cY$-shrub.
\end{itemize}


\end{theorem}

Informally speaking, this duality statement claims that for any bipartition $\cX\cup\cY$ of the ground set of a Boolean lattice there exists either a red copy of $\QQ(\cX)$ that is restricted to $\cX$ or a blue copy of the factorial tree $\cO(\cY)$ restricted to $\cY$.
This result can be seen as a strengthening of Lemma \ref{shift_lem} in the special case when we forbid a blue copy of $\Lambda$.
The Duality Theorem implies a criterion for blue/red colored Boolean lattices $Q$ to have neither a blue copy of $\Lambda$ nor a red copy of $Q_n$.

\begin{corollary}\label{cor_duality}
Let $n,k\in\N$ and $N=n+k$. Let $Q=\QQ([N])$ be a blue/red colored Boolean lattice with no blue copy of $\Lambda$.
There is no red copy of $Q_n$ in $Q$ if and only if for every $\cY\in\binom{[N]}{k}$ there exists a blue $\cY$-shrub in $Q$.
\end{corollary}

\begin{proof}
Lemma \ref{embed_lem} provides that there is a red copy of $Q_n$ in $Q$ if and only if 
there exists a partition $[N]=\cX\cup\cY$ of the ground set of $Q$ with $|\cX|=n$ and $|\cY|=k$ 
as well as an $\cX$-good embedding $\phi$ of $Q(\cX)$ into $Q$ with a monochromatic red image. 

If there is a red copy of $Q_n$ in $Q$, then for $\cX,\cY$ from Lemma \ref{embed_lem} there is also an $\cX$-good copy of $\QQ(\cX)$.
Thus by Theorem \ref{thm_duality} there is no blue $\cY$-shrub. 

On the other hand, if there is no red copy of $Q_n$ in $Q$, there is no red $\cX$-good copy of $\QQ(\cX)$ for any $\cX \in \binom{[N]}{n}$. 
Then for an arbitrary $n$-element subset $\cX$ of $[N]$, let $\cY= [N]\setminus \cX$. Now
Theorem \ref{thm_duality} implies that there exists a blue $\cY$-shrub.  In particular, there is a blue $\cY$-shrub for any $k$-element subset $\cY$ of $[N]$. 
\end{proof}
\bigskip

Throughout the section, let $\cX$ and $\cY$ be fixed disjoint sets. 
Let $Q=\QQ(\cX\cup\cY)$ be a Boolean lattice on ground set $\cX\cup\cY$. We fix an arbitrary blue/red coloring of $Q$ with no blue copy of $\Lambda$.
We always let $n,k\in\N$ such that $|\cX|=n$, $|\cY|=k$ and let $N=n+k$. 
For $X\subseteq \cX$, $Y\subseteq\cY$, we usually denote the vertex $X\cup Y$ by $(X,Y)$.
\\

In order to characterize colorings of $Q$ which do not contain an embedding $\phi$ of $\QQ(\cX)$ into $Q$ such that for every $X\in \QQ(\cX)$, $\phi(X)$ is red and $\phi(X)\cap \cX=X$,
we introduce the following notation.

For $X\subseteq \cX$ and $Y\subseteq \cY$, we say that the vertex $(X,Y)\in Q$ is \textit{embeddable} if there is an embedding $\phi:\QQ(\cX)\cap\{X'\subseteq\cX\colon X'\supseteq X\}\to Q$ 
with a monochromatic red image, such that $\phi(X')\cap \cX=X'$ for all $X'$ and $\phi(X)\supseteq (X,Y)$. 
We say that $\phi$ \textit{witnesses} that $(X,Y)$ is embeddable.

This definition immediately implies:

\begin{observation}\label{lem2}
$(\varnothing,\varnothing)$ is not embeddable if and only if there is no embedding $\phi:\QQ(\cX)\to Q$ such that for every $X'\subseteq \cX$, $\phi(X')$ red and $\phi(X')\cap \cX=X'$.
\end{observation}

The key ingredient for the proof of the Duality Theorem, Theorem \ref{thm_duality}, is the following characterization of embeddable vertices.

\begin{lemma}\label{lem3}
Let $X\subseteq\cX$, $Y\subseteq \cY$. Let $Q= \QQ (\cX\cup \cY)$ be a blue/red colored Boolean lattice with no blue copy of $\Lambda$. Then $(X,Y)$ is embeddable if and only if either
\begin{enumerate}[label=(\roman*)]
\item $(X,Y)$ is blue and there is a $Y'\subseteq\cY$ with $Y'\supsetneq Y$ such that $(X,Y')$ is embeddable, or
\item $(X,Y)$ is red and for all $X'\subseteq\cX$ with $X'\supsetneq X$, \  $(X',Y)$ is embeddable.
\end{enumerate}
\end{lemma}

Note that if $X\subseteq\cX$ and $(X,\cY)$ is blue, then $(X,\cY)$ is not embeddable.

\begin{proof}

First suppose that $(X,Y)$ is embeddable. Let $\phi$ be an embedding of $\QQ(\cX)\cap\{X'\subseteq\cX\colon X'\supseteq X\}$ into $Q$ witnessing that $(X,Y)$ is embeddable. 

If $(X,Y)$ is blue, then $\phi(X)\supsetneq(X,Y)$ because $\phi$ has a monochromatic red image. 
Thus there exists $Y'\subseteq\cY$ with $Y'\supsetneq Y$ such that $\phi(X)=(X,Y')$.
But then $\phi$ also witnesses that $(X,Y')$ is embeddable, so Condition (i) is fulfilled.

If $(X,Y)$ is red, pick some arbitrary $X^*\subseteq\cX$ such that $X^*\supsetneq X$. 
Then the function $\phi^*:\QQ(\cX)\cap\{X'\subseteq\cX\colon X'\supseteq X^*\}\to Q$, $\phi^*(X')=\phi(X')$ is a restriction of $\phi$ 
and therefore an embedding with a monochromatic red image such that $\phi^*(X')\cap \cX=X'$ for all $X'$ and $\phi^*(X^*)\supseteq (X^*,Y)$. 
Thus for every $X^*\subseteq\cX$ with $X^*\supsetneq X$, the vertex $(X^*,Y)$ is embeddable, i.e., Condition (ii) is fulfilled.\\

Now, suppose that Condition (i) or Condition (ii) hold. 
If (i) holds, then $(X,Y)$ is blue and there is some $Y'\supsetneq Y$ such that $(X,Y')$ is embeddable. Then the embedding witnessing that also verifies that $(X,Y)$ is embeddable.\\

For the rest of the proof we assume that (ii) holds, i.e., that $(X,Y)$ is red and for any $X'\subseteq\cX$ with $X'\supsetneq X$ the vertex $(X',Y)$ is embeddable.
We define the required embedding $\phi: \QQ(\cX)\cap\{X'\subseteq\cX\colon X'\supseteq X\}\to Q$ for every $X'$ with $X\subseteq X'\subseteq\cX$
depending on the number of minimal $X^*$'s, $X\subseteq X^*\subseteq X'$ such that $(X^*,Y)$ is blue as follows. Let $X'$ with $X\subseteq X'\subseteq\cX$ be arbitrary.
\begin{enumerate}[label=(\arabic*)]
\item If for all $X^*$ with $X\subseteq X^*\subseteq X'$, the vertex $(X^*,Y)$ is red, let $\phi(X')=(X',Y)$. Note that this case includes $X'=X$.
\item If there is a unique minimal $X^*$ such that $X\subseteq X^*\subseteq X'$ and $(X^*,Y)$ is blue, then $(X^*,Y)$ is embeddable by Condition (ii).
Let $\phi_{X^*}$ be an embedding witnessing that. Then set $\phi(X')=\phi_{X^*}(X')$.
\item Otherwise, let $\phi(X')=(X',\cY)$. 
\end{enumerate}

Cases (1)-(3) determine a partition of 
the set $\{X'\subseteq\cX\colon X'\supseteq X\}$ into three pairwise disjoint parts.
Let $\cM_j$, $j\in[3]$, be the set of  those vertices $X'$ for which $\phi$ was assigned in Case (j). Note that $\cM_1\cup\cM_2\cup\cM_3=\{X'\subseteq\cX\colon X'\supseteq X\}$.
\\

\noindent \textbf{Claim.} ~~ The function $\phi$ witnesses that $(X,Y)$ is embeddable.
\begin{itemize}
\item Clearly, for every $X'\subseteq\cX$ with $X'\supseteq X$, $\phi(X')\cap \cX=X'$.
\item $(X,Y)$ is red, so $X\in\cM_1$. Thus $\phi(X)=(X,Y)$.
\item The argument verifying that $\phi(X')$ is red for every $X'\subseteq\cX$ with $X'\supseteq X$ depends on $i$ such that $X'\in\cM_i$.
If $X'\in\cM_1$, it is immediate that $\phi(X')$ is red. If $X'\in\cM_2$, $\phi_{X^*}$ has a monochromatic red image, thus $\phi(X')=\phi_{X^*}(X')$ is also red.
Now consider the case that $X'\in\cM_3$, i.e.\ there are $X_1,X_2\subseteq \cX$ with $X_1\neq X_2$ and $X\subseteq X_i\subseteq X'$, $i\in[2]$, 
such that $(X_i,Y)$ are blue and $X_i$ are both minimal with this property. The latter condition implies that $X_1$ and $X_2$ are incomparable,
in particular $(X_1,Y)$ and $(X_2,Y)$ are incomparable as well.
Moreover, observe that $X_i\neq X'$, $i\in[2]$, because $X'$ is by definition comparable to both $X_1$ and $X_2$. 
Now assume for a contradiction that $\phi(X')=(X',\cY)$ is blue. Recall that $(X_1,Y)$ and $(X_2,Y)$ are blue. 
We know that $X_i\subsetneql X'$ and $Y\subseteq\cY$, thus $(X_i,Y)\subsetneql (X',\cY)$.
As a consequence, $(X_1,Y)$, $(X_2,Y)$ and $(X',\cY)$ induce a blue copy of $\Lambda$ in $Q$, which is a contradiction. Thus $\phi$ has a monochromatic red image.
\item It remains to show that $\phi$ is an embedding. Let $X_1,X_2\subseteq \cX$ be arbitrary with $X\subseteq X_i\subseteq X'$, $i\in[2]$.
We shall show that $X_1\subseteq X_2$ if and only if $\phi(X_1)\subseteq \phi(X_2)$.
One direction is easy to prove: If $\phi(X_1)\subseteq \phi(X_2)$, then $X_1=\phi(X_1)\cap \cX\subseteq \phi(X_2)\cap \cX=X_2$.
Now suppose that $X_1\subseteq X_2$.
Let $Y_1,Y_2\subseteq\cY$ such that $\phi(X_1)=(X_1,Y_1)$ and $\phi(X_2)=(X_2,Y_2)$. Then we shall show that $Y_1\subseteq Y_2$.
Note that $Y\subseteq Y_i \subseteq \cY$ for $i\in[2]$.\\

Assume that at least one of $X_1$ or $X_2$ is in $\cM_1\cup\cM_3$.
If $X_1\in\cM_1$, then $Y_1=Y$ and we are done as $Y\subseteq Y_2$. 
Furthermore, if $X_2\in\cM_3$, then $Y_2=\cY$ and we are done as well since $Y_1\subseteq\cY$.
If $X_1\in\cM_3$, then $X_1\subseteq X_2$ implies that $X_2$ is also in $\cM_3$.
Conversely, if $X_2\in\cM_1$, the fact that $X_2\supseteq X_1$ yields that $X_1\in\cM_1$ and we are done as before.
\\

As a final step, suppose that $X_1,X_2\in\cM_2$. 
This implies that for each $i\in[2]$, there is a unique minimal vertex $X^*_i$ such that $X\subseteq X^*_i\subseteq X_i$ and $(X^*_i,Y)$ is blue.
Applying the initial assumption, $X^*_1\subseteq X_1\subseteq X_2$. By minimality of $X^*_2$, we obtain that $X^*_2\subseteq X^*_1$.
Now this provides that $X^*_2\subseteq X^*_1\subseteq X_1$. Using the minimality of $X^*_1$, we see that $X^*_1\subseteq X^*_2$, so $X^*_1=X^*_2$.
\\

Recall that $(X^*_1,Y)=(X^*_2,Y)$ is embeddable since $X_1,X_2\in\cM_2$. Consider the function $\phi_{X^*_1}=\phi_{X^*_2}$ witnessing that.
Because $\phi_{X^*_1}$ is an embedding and $X^*_1\subseteq X_1\subseteq X_2$, we obtain $\phi_{X^*_1}(X_1)\subseteq\phi_{X^*_1}(X_2)$.
Combining the given conditions, $$\phi(X_1)=\phi_{X^*_1}(X_1)\subseteq\phi_{X^*_1}(X_2)=\phi_{X^*_2}(X_2)=\phi(X_2),$$
that implies that  $Y_1\subseteq Y_2$.
\end{itemize}
This concludes the proof of the Claim and the Lemma.
\end{proof}

\begin{corollary}\label{cor3}
Let $X\subseteq \cX$ and $S\in\cO(\cY)$ such that $(X,\underline{S})$ is not embeddable. 
Then there exists some $X'\subseteq \cX$, $X'\supseteq X$, such that $(X',\underline{S})$ is blue and not embeddable.
\end{corollary}
\begin{proof}
If $(X,\underline{S})$ is blue, we are done. Otherwise Lemma \ref{lem3} yields an $X_1\subseteq \cX$, $X_1\supsetneq X$ such that $(X_1,\underline{S})$ is not embeddable.
By repeating this argument, we find an $X'\subseteq \cX$, $X'\supseteq X$, with $(X',\underline{S})$ is blue and not embeddable.
\end{proof}

Next we show a connection between embeddable vertices and the existence of a weak $\cY$-shrub.
Recall that a weak $\cY$-shrub is a subposet $\PP$ of $Q$ such that there is a function $\tau: \cO(\cY) \to Q$ with image $\PP$ 
such that for every $S\in\cO(\cY)$, \  $\tau(S)\cap \cY=\underline{S}$, and for every $S,T\in \cO(\cY)$ with $S<_\cO T$, \ $\tau(S)\subsetneql \tau(T)$. 

\begin{lemma}\label{lem4}
If $(\varnothing,\varnothing)$ is not embeddable, then there is a monochromatic blue weak $\cY$-shrub.
\end{lemma}
\begin{proof}
We construct $\tau:\cO(\cY)\to Q$ iteratively and increasingly with respect to the order of $\cO(\cY)$. 
Suppose that $(\varnothing,\varnothing)$ is not embeddable. 
By Corollary \ref{cor3} there is some $X_\varnothing\subseteq \cX$ such that $(X_\varnothing,\varnothing)$ is blue and not embeddable.
Let $\tau(\varnothing)=(X_\varnothing,\varnothing)$. 
From here, we continue iteratively. Suppose that for $S\in\cO(\cY)$, $\underline{S}\neq\cY$, we have defined $X_S\subseteq \cX$ such that
\begin{enumerate}[label=(\arabic*)]
\item  $X_S\supseteq X_T$ for every $T\le_\cO S$ and 
\item $\tau(S)=(X_S,\underline{S})$ is blue and not embeddable.
\end{enumerate}
Consider an arbitrary $S'\in\cO(\cY)$ such that $S<_\cO S'$ and $|S'|=|S|+1$. 
By Lemma \ref{lem3} applied for $X_S$ and $\underline{S}$, we obtain that $(X_S,\underline{S'})$ is not embeddable. 
Then Corollary \ref{cor3} yields that there is some $X_{S'}\subseteq\cX$, $X_{S'}\supseteq X_S$, such that $(X_{S'},\underline{S'})$ is blue and not embeddable.
Observe that for $T\in\cO(\cY)$ with $T\le_\cO S'$, either $T=S'$ and so $X_T=X_{S'}$, or $T\le_\cO S$ and so by $(1)$ $X_T\subseteq X_S\subseteq X_{S'}$.
Let $\tau(S')=(X_{S'},\underline{S'})$.
\\

Using this procedure, we define $\tau$ for all $S\in\cO(\cY)$. Let $\PP$ be the subposet of $Q$ induced by the image of $\tau$. 
We shall show that $\PP$ is a weak $\cY$-shrub witnessed by the function $\tau$. 
By (2), for every $S\in \cO(\cY)$, \ $\tau(S)$ is blue and $\tau(S)\cap \cY=\underline{S}$.

Let $S,T\in\cO(\cY)$ with $S<_\cO T$. Let $X_S,X_T\subseteq\cX$ such that $\tau(S)=(X_S,\underline{S})$ and $\tau(T)=(X_T,\underline{T})$.
Clearly, $\underline{S}\subsetneql \underline{T}$. Moreover,  item $(1)$ implies that $X_S\subseteq X_T$.
Consequently, $\tau(S)\subsetneql\tau(T)$. 

\end{proof}

Combining the previously presented Lemmas, we can now prove the Duality Theorem. 

\begin{proof}[Proof of Theorem \ref{thm_duality}]
Let $\cX$ and $\cY$ be disjoint sets.
Let $Q=Q(\cX\cup \cY)$ be a blue/red colored Boolean lattice which contains no blue copy of $\Lambda$. \\
First suppose that there is no red $\cX$-good copy of $\QQ(\cX)$.
By Observation \ref{lem2}, $(\varnothing,\varnothing)$ is not embeddable and Lemma \ref{lem4} provides that there is a blue weak $\cY$-shrub in $Q$.
Using Proposition \ref{obs_shrub} we obtain a blue $\cY$-shrub in $Q$.
This shows that there is either a red $\cX$-good copy of $\QQ(\cX)$ or a blue $\cY$-shrub. \\

Next we show that both events could not happen simultaneously. 
Let $n=|\cX|$, $k=|\cY|$ and $N=n+k$.
Assume that there exist both an $\cX$-good embedding $\phi\colon \QQ(\cX) \to Q$ with monochromatic red image
as well as a $\cY$-good embedding $\tau\colon \cO(\cY) \to Q$ with a monochromatic blue image.
\\

We apply an iterative argument in order to find a contradiction. Let $Y_0=\varnothing$ and let $S_0=(Y_0,\le)$ be the empty ordered set.
Now let $X_1\subseteq\cX$ such that $\tau(S_0)=(X_1,\underline{S_0})$ and let $Y_1\in\cY$ such that $\phi(X_1)=(X_1,Y_1)$.
Since $\phi(X_1)$ is red but $\tau(S_0)$ is blue, we know that $\phi(X_1)\neq\tau(S_0)$ and thus $Y_1\neq \underline{S_0}=\varnothing$, so $|Y_1|\ge 1$. 
\\

Now say that we already defined $X_1,\dots,X_i$, $Y_0,\ldots,Y_i$, $S_0,\dots,S_{i-1}$ for some $i\in[k]$ such that
\begin{itemize}
\item $S_{i-1}\in\cO$ and $\underline{S_{i-1}}=Y_{i-1}$,
\item $\tau(S_{i-1})=(X_i,\underline{S_{i-1}})$,
\item $\phi(X_i)=(X_i,Y_i)$, and 
\item $Y_{i-1}\subsetneql Y_i\subseteq \cY$ and $|Y_i|\ge i$.
\end{itemize}
Fix any ordering $S_i$ of $Y_i$ such that $S_{i-1}<_\cO S_i$. Such $S_i$ exists because $\underline{S_{i-1}}=Y_{i-1}\subsetneql Y_i$. 

Then let $X_{i+1}$ be  such that $\tau(S_i)=(X_{i+1},\underline{S_i})$.
Since $S_{i-1}<_\cO S_i$ and $\tau$ is an embedding, $(X_{i},\underline{S_{i-1}})=\tau(S_{i-1})\subseteq\tau(S_i)=(X_{i+1},\underline{S_i})$, therefore $X_i\subseteq X_{i+1}$.
Note that $\phi(X_i)=(X_i,\underline{S_i})$ is colored red but $\tau(S_i)=(X_{i+1},\underline{S_i})$ is blue.
Therefore $X_{i}\neq X_{i+1}$, consequently $X_i\subsetneql X_{i+1}$ and in particular $\phi(X_i)\subsetneql\phi(X_{i+1})$ because $\phi$ is an embedding.

Next let $Y_{i+1}\subseteq\cY$ such that $\phi(X_{i+1})=(X_{i+1},Y_{i+1})$. 
Then $Y_{i+1}\supseteq Y_i$ and furthermore, because $(X_{i+1},Y_i)$ is blue but $\phi(X_{i+1})$ is red, $Y_{i+1}\neq Y_i$.
Consequently $Y_{i+1}\supsetneq Y_i$, and in particular $|Y_{i+1}|\ge|Y_i|+1\ge i+1$.
\\

Iteratively, we obtain $Y_{k+1}\subseteq \cY$ with $|Y_i|\ge k+1$, a contradiction to $|\cY|=k$.
\end{proof}

\section{Random coloring with many blue shrubs} \label{sec_random}

We shall provide a coloring that will give us a lower bound on $R(\Lambda, Q_n)$. Note that we do not provide an explicit construction but only prove the existence of such a coloring. 

\begin{theorem}\label{thm_LB}
Let $N\in\N$ be sufficiently large and $k=\frac{10}{216}\frac{N}{\ln(N)}$. Consider the Boolean lattice $Q=\QQ([N])$. 
Then for sufficiently large $N$, there exists a blue/red coloring of $Q$ which contains no blue copy of $\Lambda$ and 
such that for each $\cY\in\binom{[N]}{k}$, there is a blue $\cY$-shrub in $Q$.
\end{theorem}

\begin{proof}[Proof of Theorem \ref{thm_LB}]
Let $\alpha= 21.6$ and $\beta=0.134$.
Let $N\in\N$ and $k=\frac{1}{\alpha}\frac{N}{\ln(N)}$, let $Q=\QQ([N])$.

The idea of the proof is to construct a  $\cY$-shrub, denoted $\PP_{\cY}$,  for every $\cY\in\binom{[N]}{k}$, with an additional property so that the selected shrubs are independent. 
Since each shrub does not contain a copy of $\Lambda$, it implies that the independent union of all the $\PP_{\cY}$'s also does not contain a copy of $\Lambda$.
We obtain these shrubs by randomly choosing a \textit{$\cY$-framework} for every $\cY\in\binom{[N]}{k}$ and then constructing a $\cY$-shrub based on each of them. 
Afterwards we define a coloring where every vertex in each constructed shrub is colored blue and the remaining vertices red.
\\

\noindent A \textit{$\cY$-framework} of $\cY\in\binom{[N]}{k}$ is a $4$-tuple $(\cY,A_\cY,Z_\cY, X_\cY)$ such that
\begin{itemize}
\item $\cY$, $A_\cY$ $Z_\cY$ are pairwise disjoint and $\cY\cup A_\cY\cup Z_\cY=[N]$,
\item $|A_\cY|=\tfrac{3}{2}k\ln k-k$,
\item $X_\cY\subseteq Z_\cY$.
\end{itemize}

\noindent A $\cY$-framework is \textit{random} if
\begin{itemize}
\item $A_\cY\in\binom{[N]\backslash\cY}{\tfrac{3}{2}k\ln k-k}$ is chosen uniformly at random, and
\item each element of $Z_\cY=[N]\backslash(\cY\cup A_\cY)$ is included in $X_\cY$ independently at random with probability $\frac{1}{2}$.
\end{itemize}

Now draw a random $\cY$-framework for every $\cY\in\binom{[N]}{k}$.
Observe that by choice of $k$, $k\ln  k =\frac{N}{\alpha}\cdot \frac{\ln(N)-\ln(\alpha)-\ln\ln(N)}{\ln(N)}$, so $\frac{20N}{21\alpha} \le k\ln k\le \frac{N}{\alpha}$. 
Since  $|Z_\cY|= N-\tfrac{3}{2}k\ln k$,  we have  $(1-\frac{3}{2\alpha})N\le|Z_\cY|\le (1-\frac{10}{7\alpha})N$.
\\

\noindent \textbf{Claim 1.} ~~ W.h.p.\ for every $\cY_1,\cY_2\in\binom{[N]}{k}$ with $\cY_1\neq\cY_2$, \ $|X_{\cY_1}\cap Z_{\cY_2}|\ge \beta N$.\\
Consider some arbitrary $\cY_1,\cY_2\in\binom{[N]}{k}$, $\cY_1\neq\cY_2$. Observe that $(1-\frac{3}{\alpha})N\le|Z_{\cY_1}\cap Z_{\cY_2}|\le|Z_\cY|\le(1-\frac{10}{7\alpha})N$.
In a random $\cY$-framework, each element of $Z_{\cY_1}\cap Z_{\cY_2}$ is contained in $X_{\cY_1}\cap Z_{\cY_2}$ independently with probability $\frac12$.
Consequently, $|X_{\cY_1}\cap Z_{\cY_1}|\sim \text{Bin}(|Z_{\cY_1}\cap Z_{\cY_2}|,\frac12)$ and $\EEE(|X_{\cY_1}\cap Z_{\cY_1}|)=\frac12 |Z_{\cY_1}\cap Z_{\cY_2}|$.   
We have $1-\frac3\alpha\ge2\beta$. In addition, $\frac{\big(\frac{1}{2}-\frac{3}{2\alpha}-\beta\big)^2}{1-\frac{10}{7\alpha}}>\frac{2}{\alpha}$, thus  there exist some $\epsilon>0$ such that $\frac{\big(\frac{1}{2}-\frac{3}{2\alpha}-\beta\big)^2}{1-\frac{10}{7\alpha}}\geq \epsilon + \frac{2}{\alpha}$. 
Applying Chernoff's inequality gives
\begin{eqnarray*}
\PPP(|X_{\cY_1}\cap Z_{\cY_2}|\le \beta N)&=& \PPP\left(|X_{\cY_1}\cap Z_{\cY_2}|\le\frac{|Z_{\cY_1}\cap Z_{\cY_2}|}{2}-\left(\frac{|Z_{\cY_1}\cap Z_{\cY_2}|}{2}-\beta N\right)\right)\\
&\le & \exp\left(-\frac{\big(\frac{|Z_{\cY_1}\cap Z_{\cY_2}|}{2}-\beta N\big)^2}{|Z_{\cY_1}\cap Z_{\cY_2}|}\right)\\
&\le & \exp\left(-\frac{\big((\frac{1}{2}-\frac{3}{2\alpha})-\beta\big)^2}{(1-\frac{10}{7\alpha})}\cdot N\right)\\
&\le & \exp\left(-\left(\frac{2}{\alpha}+\epsilon\right)\cdot N\right).
\end{eqnarray*}

\noindent
Let $E_1$ be  the event that for some distinct $\cY_1,\cY_2\in\binom{[N]}{k}$,  $|X_{\cY_1}\cap Z_{\cY_2}|\le \beta N$. Then 
\begin{eqnarray*}
\PPP(E_1)&= & \binom{N}{k}\left(\binom{N}{k}-1\right) \PPP(|X_{\cY_1}\cap Z_{\cY_2}|\le \beta N)\\
&\le &  N^{2k} \exp\left(-\left(\frac{2}{\alpha}+\epsilon\right)\cdot N\right)\\
&\le & \exp\left(\frac{2N\ln (N)}{\alpha\ln  (N)} - \left(\frac{2}{\alpha}+\epsilon\right)\cdot N\right)\\
&=& \exp(-\epsilon N)\to 0, \mbox{ as }  N\to\infty.
\end{eqnarray*}
This proves Claim 1.\\

\noindent \textbf{Claim 2.} ~~ W.h.p.\ for every $\cY_1,\cY_2\in\binom{[N]}{k}$ with $\cY_1\neq\cY_2$, \ $X_{\cY_1}\cap Z_{\cY_2}\not\subseteq X_{\cY_2}$.\\
We can suppose that the collection of random frameworks fulfills the property of Claim 1.
Let $\cY_1,\cY_2\in\binom{[N]}{k}$ be such that $\cY_1\neq\cY_2$.
Note that each element of $X_{\cY_1}\cap Z_{\cY_2}$ is contained in $X_{\cY_2}$ with probability $\frac{1}{2}$. Thus,
$$\PPP(X_{\cY_1}\cap Z_{\cY_2}\subseteq X_{\cY_2})=\left(\frac{1}{2}\right)^{|X_{\cY_1}\cap Z_{\cY_2}|}\le 2^{-\beta N}.$$
Let $E_2$ be the event that  there exist $\cY_1,\cY_2\in\binom{[N]}{k}$ with $\cY_1\neq\cY_2$ such that $X_{\cY_1}\cap Z_{\cY_2}\not\subseteq X_{\cY_2}$. 
Since $\frac{2}{\alpha}<\ln (2)\beta$, we have 
$$\PPP(E_2)\le 
N^{2k} \PPP(X_{\cY_1}\cap Z_{\cY_2}\subseteq X_{\cY_2})
\le N^{2k}\cdot 2^{-\beta N}=\exp\left(\frac{2}{\alpha}N-\ln (2)\beta N\right)\to 0,  \mbox{ as } N\to\infty.$$
This proves Claim 2.\\

In particular, there exists a collection of $\cY$-frameworks $(\cY,A_\cY,Z_\cY, X_\cY)$, $\cY\in\binom{[N]}{k}$, such that 
for every $\cY_1,\cY_2\in\binom{[N]}{k}$ with $\cY_1\neq\cY_2$, \ $X_{\cY_1}\cap Z_{\cY_2}\not\subseteq X_{\cY_2}$.
\\

Note that $|A_{\cY}|=\tfrac{3}{2}k\ln k-k\ge k(\log k + \log\log k)$.
Let $\PP'_{\cY}$ be a $\cY$-shrub in $\QQ(A_{\cY}\cup \cY)$ as guaranteed by Lemma \ref{canonical-shrub}.
Note that $\PP_{\cY}$'s are not necessarily independent. 
Let $\PP_{\cY}$ be obtained from $\PP'_{\cY}$ by replacing each vertex  $W$ of $\PP'_{\cY}$ with $W\cup X_{\cY}$.
Then $\PP_{\cY}$ is a $\cY$-shrub in $Q$.  \\

\noindent \textbf{Claim 3.} ~~ Let $\cY_1, \cY_2$ be two distinct $k$-element subsets of $[N]$. Then $\PP_{\cY_1}$  and $\PP_{\cY_2}$ are  independent.\\
Consider arbitrary elements $U_i\in \PP_{\cY_i}$, $i\in[2]$. Recall that $X_{\cY_1}\cap Z_{\cY_2}\not\subseteq X_{\cY_2}$, which implies that 
there exists some $z\in (X_{\cY_1}\cap Z_{\cY_2})\backslash X_{\cY_2}$. Note that $z\in U_1$ since $X_{\cY_1}\subseteq U_1$. Moreover 
$z\not\in U_2$ since $z\in Z_{\cY_2}\setminus X_{\cY_2}$ and $(Z_{\cY_2}\setminus X_{\cY_2})\cap U_2 = \varnothing$. 
In particular, $z\in U_1\backslash U_2$. Similarly, there is an element $w\in U_2\backslash U_1$. Thus $U_1\not\sim U_2$.  \qed\\
\bigskip

We consider the following coloring $c: Q\to \{\text{blue, red}\}$. For $X\subseteq [N]$, let
\begin{equation} \nonumber
c(X) = 
 \begin{cases}
\text{blue},  		\hspace*{0.75cm}\mbox{ if } ~~  X\in \bigcup_{\cY\in\binom{[N]}{k}} \PP_\cY,\\
\text{red}, 		\hspace*{0.8cm}\mbox{ otherwise.}
\end{cases}
\end{equation}

Note that for every $\cY\in\binom{[N]}{k}$, $\PP_\cY$ witnesses that there is a blue $\cY$-shrub in $Q$.
Recall that a $\cY$-shrub is an up-tree. Applying Claim 3 the blue subposet of $Q$ is a collection of independent up-trees.
Then Lemma \ref{strucL} provides that the coloring $c$ does not contain a blue copy of $\Lambda$.
\end{proof}

\section{Proof of Theorem \ref{thm_main} and Theorem \ref{thm-MAIN}}\label{sec_proof}

\begin{proof}[Proof of Theorem \ref{thm_main}]
~\\
\textbf{ Upper Bound:~} Let $k=(1+\epsilon)\frac{n}{\log(n)}$ and consider an arbitrary blue/red colored Boolean lattice $Q$ on ground set $[n+k]$ with no blue copy of $\Lambda$.
Pick any $\cY\in\binom{[n+k]}{k}$ and assume that there is a blue $\cY$-shrub in $Q$. 
Recall that the maximal elements of the $\cY$-shrub form an antichain of size $k!$.
Sperner's theorem provides that the largest antichain in $Q$ has size $\binom{n+k}{\lfloor\frac{n+k}{2}\rfloor}$, 
so $k!\le \binom{n+k}{\lfloor\frac{n+k}{2}\rfloor}\leq  2^{n+k}$.
\\

We also have that  $k!>\left(\tfrac{k}{e}\right)^k=2^{k(\log k - \log e)}.$
By the choice of $k$, we obtain for sufficiently large $n$,
$$k\log k\ge\tfrac{(1+\epsilon)n}{\log n}\big(\log(n)-\log\log(n)\big)>\big(1+\tfrac{\epsilon}{2}\big)n.$$
In particular for sufficiently large $n$, 
$k\log k -k\log e>n+k,$
a contradiction.
Thus $Q$ does not contain a blue $\cY$-shrub for this fixed $\cY$. Then Corollary \ref{cor_duality} yields that there is a red copy of $Q_n$ in $Q$.
Consequently, each blue/red colored Boolean lattice of dimension $n+k$ contains either a blue copy of $\Lambda$ or a red copy of $Q_n$.\\

\noindent
\textbf{Lower Bound:~} Let $N$ sufficiently large, let $k=\frac{10}{216}\frac{N}{\ln(N)}$ and $n=N-k$. Note that $k\le \frac{N}{2}$, thus $n\le N\le 2n$. Let $Q=\QQ([N])$.
By Theorem \ref{thm_LB} there exists a coloring of $Q$ with no blue copy of $\Lambda$ such that for every $\cY\in\binom{[N]}{k}$, there is a blue $\cY$-shrub.
By Corollary \ref{cor_duality}, there is no red copy of $Q_n$ in this coloring, thus $R(\Lambda,Q_n)\ge N= n+ k$.
It remains to bound $k$ in terms of $n$.
Indeed, $$k=\frac{10}{216}\cdot\frac{N}{\ln(N)}\ge \frac{10}{216}\cdot\frac{n}{\ln(2n)}
= \frac{10}{216}\cdot\frac{\log(e)n}{\log(2n)}\ge \frac{1}{15}\cdot\frac{n}{\log(n)},$$
which concludes the proof.
\end{proof}
\bigskip

\begin{proof}[Proof of Theorem \ref{thm-MAIN}]
The lower bound on $R(P, Q_n)$ for $P$ containing either $\Lambda$ or $V$ follows from Theorem \ref{thm_main}. 

Consider now a poset $P$ that contains neither a copy of $\Lambda$ nor a copy of $V$.  
By Corollary \ref{strucI}, $P$ is a union of independent chains.
Assume that $P$ has $k$ independent chains on at most $\ell$ vertices each. Let $K$ be an even  integer such that $\binom{K}{K/2} \geq k$. 
Let $\cY$ be a set of size $K$ and let $\cX$ be a set, disjoint from $\cY$ of size $n+ \ell$.
Consider an arbitrary coloring of $\QQ(\cX\cup \cY)$.  
Assume that there is no red copy of $Q_n$. We shall show that there is a blue copy of $P$.\\

Let $Y_1, \ldots, Y_k$ form an antichain in $\QQ(\cY)$, its existence is guaranteed by Sperner's theorem.
Let $Q^i$ be a copy of $\QQ(\cX)$ obtained as an image of an embedding $\phi_i: \QQ(\cX) \to  \QQ(\cX\cup Y_i)$, $\phi_i(X)=X\cup Y_i$ for any $X\subseteq \cX$. 
Consider the blue vertices in $Q^i$. 
If there is no blue chain on $\ell$ vertices in $Q^i$, Corollary \ref{shift_cor} implies the existence of a red copy of $Q_n$ in $Q^i$, a contradiction. 
Thus for every $i\in[k]$, there is a blue copy $P_i$ of a chain on $\ell$ vertices in $Q^i$.  
Note that for any $A\in Q^i, B\in Q^j$, $i\neq j$, \ $A\not\sim B$, since $A\cap\cY =Y_i\not\sim Y_j=B\cap\cY$.
Thus the $P_i$'s are independent chains on $\ell$ vertices each.  Their union contains a copy of $P$.  This shows that  $R(P, Q_n) \leq n + K+\ell = n+ f(P)$. 
\end{proof} 
\medskip

\noindent \textbf{Acknowledgments:}~  The authors would like to thank Alexander Riasanovsky for comments on the manuscript. The research was partially supported by DFG grant FKZ AX 93/2-1.


\end{document}